\documentclass[12pt,letterpaper]{amsart}
\usepackage{amsmath,amssymb,amsfonts}
\usepackage{color,url,setspace,mathpazo,geometry}
\newtheorem{theorem}{Theorem}
\newtheorem{proposition}{Proposition}
\newtheorem*{theorem*}{Theorem}
\newtheorem{lemma}{Lemma}
\newtheorem{corollary}{Corollary}
\theoremstyle{remark}
\newtheorem{remark}{Remark}
\newtheorem{example}{Example}
\newcommand{\C}{\mathbb{C}}
\newcommand{\D}{\Omega}
\newcommand{\Dc}{\overline{\Omega}}
\newcommand{\dbar}{\overline{\partial}}

\title[Localization of compactness of Hankel operators]{Localization of
compactness of Hankel operators on pseudoconvex domains}

\author{S\"{o}nmez \c{S}ahuto\u{g}lu}
\address{University of Toledo, Department of Mathematics \& Statistics, 
Toledo, OH 43606, USA}
\email{sonmez.sahutoglu@utoledo.edu}

\thanks{The author is supported in part by the University of Toledo's Summer Research
Awards and Fellowships Program}

\begin{document}

\begin{abstract}
We prove the following localization  for compactness of Hankel operators on Bergman 
spaces. Assume that  $\D$ is a bounded pseudoconvex domain in $\C^{n}, p$ is a 
boundary point of $\D,$ and $B(p,r)$ is a ball centered at $p$ with radius $r$ so that
$U=\D\cap B(p,r)$ is connected. We show that if the Hankel operator $H^{\D}_{\phi}$ with
symbol $\phi \in C^1(\Dc)$ is compact  on $A^{2}(\D)$ then $H^{U}_{R_{U}(\phi)}$ is
compact on $A^{2}(U)$ where $R_U$ denotes the restriction operator on $U,$ and $A^2(\D)$
and $ A^2(U)$ denote the Bergman spaces on $\D$ and $U,$ respectively. 
\end{abstract}

\maketitle

Let $V$ be a domain in $\C^n$ and $ A^2(V)$ denote the Bergman space on $V$, the space
of square integrable holomorphic functions on $V$ with respect to the Lebesgue
measure $d\lambda$ in $\C^n.$ Let $P^V$ denote the Bergman projection, the
orthogonal projection from $L^2(V)$ onto $A^2(V).$  The Hankel operator,
$H^V_{\phi},$ with symbol $\phi\in L^{\infty}(V)$ is defined as
$H^V_{\phi}(f)=\phi f-P^V(\phi f)$ for $f\in A^2(V).$

A Hankel operator is the  commutator $[M_{\phi},P^V]$ of a multiplication
operator  with the Bergman projection.  Such commutators play important roles in some
problems in several complex variables (see, for example, \cite{CatlinD'Angelo97}).

Compactness is an important concept in analysis. In this paper, we are interested 
in the localization of compactness of Hankel operators. More precisely, we are 
interested in the following question:

\textit{Let $\D$ be a bounded pseudoconvex domain in 
$\C^n, \phi\in L^{\infty}(\D),$ and $p\in b\D$ where $b\D$ denotes the boundary of  $\D.$
Assume that  $U=\D\cap B(p,r)$ is connected, $R_U$ denotes the restriction
onto $U,$ and $H^{\D}_{\phi}$ is compact on $A^2(\D).$ Is
$H^{U}_{R_U(\phi)}$ compact on $A^2(U)$?}

We are not able to answer the question in general.  Using the $\dbar$-Neumann operator, we
show that the answer is yes when the symbol is $C^1$ on the closure of the domain. For 
more information about the $\dbar$-Neumann problem see 
\cite{ChenShawBook,StraubeBook} and consult \cite{ZhuBook} about the theory of 
Hankel operators on domains in $\C$. 

It would be interesting to know if Theorem  \ref{ThmMain} below is still true
without the $C^1$ differentiability requirement.  We note that in dimension one,  
regularity of the symbol can be relaxed. For example, one can choose the symbol
to be continuous up to the boundary  (see Proposition \ref{PropInC} below). 
However, in that case localization is trivial as compactness is not due to localization. 
The following proposition is probably known, although we cannot provide a
reference. We therefore include a proof that was suggested in \cite{StraubePrivate}.

\begin{proposition}\label{PropInC}
Let $\D$ be a bounded domain in $\C$ and $\phi\in C(\Dc).$ Then the Hankel
operator $H^{\D}_{\phi}$ is compact on $A^2(\D).$ 
\end{proposition}

The main result of this paper is the following theorem.

\begin{theorem}\label{ThmMain}
Let $\D$ be a bounded pseudoconvex domain in $\C^{n}, p\in b\D,$ and
$B(p,r)$ be a ball centered at $p$ with radius $r>0$ so that $U=\D\cap B(p,r)$
is connected. Assume that $\phi \in C^1(\Dc)$ and $H^{\D}_{\phi}$  is compact  on
$A^{2}(\D).$ Then $H^{U}_{R_{U}(\phi)}$ is compact on $A^{2}(U) .$
\end{theorem}

We note that in the theorem above no regularity of $b\D$ is assumed. That is, the boundary
of $\D$ may be very irregular. Also $\phi\in C^1(\Dc)$ means that the function $\phi$ and
all of its first partial derivatives  have continuous extensions up to the boundary.

Localization is an important technique in analysis. So we believe that such
results can be useful in studying compactness of Hankel operators in connection to
boundary geometry (see, for example, \cite{CelikSahutoglu,CuckovicSahutoglu09}).
This particular localization can be useful in the following way: when one
studies compactness of Hankel operators in relation to the boundary geometry of
a smooth bounded pseudoconvex domain, usually a local holomorphic  change
of coordinates is needed to simplify the boundary geometry while preserving the
compactness of the operator. Theorem \ref{ThmMain} guarantees that this is
possible when the local domain is an intersection with a ball and the symbol is
sufficiently regular.  

The converse of Theorem \ref{ThmMain} is known to be true (see, for
example, ii. in Proposition 1 in \cite{CuckovicSahutoglu09}). Hence we have the
following  corollary.

\begin{corollary}
Let $\D$ be a bounded pseudoconvex domain in $\C^{n},\phi \in C^1(\Dc),$
and $B(q,r)$ denote a ball centered at $q\in b\D$ with radius $r>0.$ 
\begin{itemize}
 \item[i.] If $U=\D\cap B(p,r)$ is connected for some $p\in b\D,r>0,$ and 
$H^{\D}_{\phi}$ is compact on $A^{2}(\D)$ then $H^{U}_{R_{U}(\phi)}$ is compact on
$A^{2}(U).$
\item[ii.] Assume that for any $p\in b\D$ there exists $r>0$ such that 
$U=\D\cap B(p,r)$ is connected and  $H^{U}_{R_{U}(\phi)}$ is compact on
$A^{2}(U).$ Then $H^{\D}_{\phi}$ is compact  on $A^{2}(\D).$
\end{itemize}
\end{corollary}

\begin{remark}\label{remark}
The proof of Theorem \ref{ThmMain} shows that the localization of compactness
of Hankel operators is still true on the intersection of the domain $\D$ with
strongly pseudoconvex domains. Whether Theorem \ref{ThmMain} 
holds on the intersection of $\D$ with domains with compact $\dbar$-Neumann
operator is still open. However, it may not hold on the intersection
of $\D$ with a general pseudoconvex domain. For example, let $U=\D\cap V$ where
$V$ is a smooth bounded convex domain and $bV\cap \D$ contains a nontrivial
analytic disc $D$, and  $\phi\in C^{\infty}(\Dc)$ such that $\phi\equiv 0$ on
$b\D$ and $\phi \circ \beta$ is not holomorphic for some holomorphic mapping 
$\beta:\{z\in \C:|z|<1\}\to D$. Then one can use the facts that the product
operator $M_{\phi}:A^2(\D)\to L^2(\D)$ is compact and the Hankel operator
$H^{\D}_{\phi}$ is a composition of the projection on the orthogonal complement
of the Bergman space with $M_{\phi}$ to show that $H^{\D}_{\phi}$ is compact.
Moreover, since $\phi \circ \beta$ is not holomorphic for some holomorphic mapping 
$\beta:\{z\in \C:|z|<1\}\to D$ Theorem 2 in  \cite{CuckovicSahutoglu09} implies
that $H^U_{R_U(\phi)}$ is not compact (even though, 
\cite[Theorem 2]{CuckovicSahutoglu09} is stated for smooth domains its proof is
still valid on $U$). Therefore, $H^{\D}_{\phi}$ is compact on $A^2(\D)$ while
$H^U_{R_U(\phi)}$ is not compact on $A^2(U).$
\end{remark}

In the following examples we show that boundedness and pseudoconvexity of the domain 
are necessary  in Theorem \ref{ThmMain}.

\begin{example}\label{Example1}
This example shows that boundedness of the domain $\D$ is necessary. Let us
denote $\mathbb{D} =\{z\in \C:|z|<1\}, \D=\mathbb{D}\times \C, p=(1,0),$ and 
$ \phi(z,w)=\xi(|w|)$ where $\xi\in C^{\infty}_{0}(-1,1)$ and $\xi(0)=1.$ Let
$f\in A^{2}(\D)$ then 
\[\int_{\D}|f(z,w)|^{2}d\lambda(z,w)=\int_{\mathbb{D}}\int_{\C}|f(z,w)|^{2}
d\lambda(w)d\lambda(z)<\infty.\]
Fubini's theorem implies that the set 
$\Gamma=\{z\in \mathbb{D}:\int_{\C}|f(z,w)|^{2}d\lambda(w)=\infty\}$ has measure
zero. Hence $f(z,w)=0$ for $z\not\in \Gamma$ and $w\in \C.$  This implies that
$A^{2}(\D)=\{0\}$ and $H_{\phi}=0.$ In particular, $H_{\phi}$ is compact.
However, since there is an analytic disc through $p$  in the boundary of
$U=\D\cap B(p,1)$   \cite[Theorem 1]{CuckovicSahutoglu09} (see  the last
sentence in Remark \ref{remark}) implies that  the operator 
$H_{R_{U}(\phi)}^{U}$ is not compact on $A^2(U).$ 
\end{example}

\begin{example}
This example shows that pseudoconvexity of the domain is necessary 
(for more information on pseudoconvexity see \cite{KrantzBook,RangeBook}).
In  \cite{CelikSahutoglu} \c{C}elik and the author constructed an annulus
type domain $\D\subset \C^3$ (that is, $\D=\D_1\setminus \Dc_2$ where 
$\Dc_2\subset \D_1,$ and $\D_1$ and $\D_2$ are smooth bounded pseudoconvex
domains) such that $H_{\phi}^{\D}$ is compact on $A^2(\D)$ for all 
$\phi\in C(\overline{\D}).$ However, they show that  there exist $p\in b\D,$ on
the inner boundary of $\D,$ and $r>0$ such that $U=\D \cap B(p,r)$ is a convex
domain and there exists a disc through $p$ in the boundary of $U.$ Hence $N^U$
is not compact (see \cite[Theorem 1.1]{FuStraube98}). Furthermore, 
there exists  $\phi\in C^{\infty}(\overline{U})$ such that  $H^{U}_{\phi}$ is
not compact on $A^2(U)$ because, on a convex domain $V$, the Hankel operator
$H^V_{\phi}$ is compact for all $\phi \in C^{\infty}(\overline{V})$ if and only
if $N^V$ is compact (see\cite{FuStraube98}). 
\end{example}

\section*{Proof of Theorem \ref{ThmMain} and Proposition \ref{PropInC}}
We use the $\dbar$-Neumann problem in the proof of Theorem \ref{ThmMain}.
Let $\D$ be a bounded pseudoconvex domain in $\C^n$ and 
$\Box^{\D}=\dbar\dbar^*+\dbar^*\dbar$ be defined on square integrable $(0,1)$-forms,
$L_{(0,1)}^{2}(\D),$ where $\dbar^*$ is the Hilbert space adjoint of $\dbar.$ 
Kohn \cite{Kohn63} and H\"ormander \cite{Hormander65} showed that
(since $\D$ is a pseudoconvex domain) $\Box$ has a solution operator, denoted by
$N^{\D},$ on $L_{(0,1)}^{2}(\D).$ Kohn \cite{Kohn63} also showed that 
$P^{\D}=I-\dbar^* N^{\D}\dbar.$ Therefore, $H^{\D}_{\phi}(f)=\dbar^*N^{\D}(f\dbar\phi)$
for  $f\in A^2(\D)$ and $\phi\in C^1(\Dc).$ We note that $H^{\D}_{\phi}(f)$ is 
the canonical solution for  $\dbar u=f\dbar\phi.$ That is, $H^{\D}_{\phi}(f)$ is the
solution that is orthogonal to $A^2(\D)$ (or equivalently, it is the solution with the 
smallest norm in $L^2(\D)$). We refer the reader to \cite{ChenShawBook,StraubeBook} and
\cite{CuckovicSahutoglu09} (and references therein) for more information about the
$\dbar$-Neumann problem and compactness of Hankel operators on Bergman spaces.  

We use a series of Lemmas for the proof of Theorem \ref{ThmMain}. 
We note that the following Lemma is an immediate corollary of 
\cite[Proposition V.2.3]{D`AngeloIneqBook} (see also \cite[Lemma 4.3]{StraubeBook}). 
\begin{lemma}\label{LemCompEstimate}
Let $T:X\to Y$ be a linear operator between two Hilbert spaces $X$ and $Y$. Then
$T$ is compact if and only if  for every $\varepsilon>0$ there exist a compact
operator $K_{\varepsilon}:X\to Y$ so that 
\[\|T(h)\|_Y\leq \varepsilon\|h\|_X+\|K_{\varepsilon}(h)\|_Y \text{ for }
h\in X.\]
\end{lemma}

In the proof of Theorem \ref{ThmMain} we will need to apply Lemma
\ref{LemCompEstimate} in the following set-up.

\begin{lemma} \label{LemCompactH}
Let $\D$ be a bounded pseudoconvex domain in $\C^n,$ $\phi\in C^1(\Dc),$ and 
$X_{\phi}(\D)$ be the closure of $\{f\dbar \phi \in L^2_{(0,1)}(\D): f\in A^{2}(\D)\} $
in $L^2_{(0,1)}(\D).$ Then 
$H^{\D}_{\phi}$ is compact on $A^2(\D)$ if and only if for every
$\varepsilon>0$ there exists a compact operator 
$K_{\varepsilon}:X_{\phi}(\D)\to L^{2}(\D)$  such that 
\begin{equation}\label{CompEst}
\|\dbar^{*}N^{\D}(f\dbar\phi)\|\leq \varepsilon\|f\dbar
\phi\|+\|K_{\varepsilon}(f\dbar \phi)\| \text{ for all } f\in A^{2}(\D). 
\end{equation}
\end{lemma}

\begin{proof} Assume that $H^{\D}_{\phi}$ is compact on $A^{2}(\D).$ Then 
$\dbar^{*}N^{\D}$ is compact on a dense subset of $X_{\phi}(\D)$ which implies that it is 
 compact on  $X_{\phi}(\D).$  Then applying Lemma \ref{LemCompEstimate} with 
$T= \dbar^{*}N^{\D}$ and $X=X_{\phi}(\D)$ we get the
following estimate: for every $\varepsilon>0$ there exists a compact operator 
$K_{\varepsilon}:X_{\phi}(\D) \to L^{2}(\D)$ so that 
\[\|\dbar^{*}N^{\D}(f\dbar \phi)\|\leq \varepsilon\|f\dbar \phi\|
+\|K_{\varepsilon}(f\dbar \phi)\| \textrm{ for } f \in A^2(\D).\]
On the other hand, if we assume that we have \eqref{CompEst} then Lemma
\ref{LemCompEstimate} implies that $\dbar^{*}N^{\D}$ is a compact operator on
$X_{\phi}(\D).$ Hence,  $H^{\D}_{\phi}$ is compact on $A^2(\D).$ This completes
the proof of Lemma \ref{LemCompactH}.
\end{proof}

The following famous theorem of H\"{o}rmander \cite[Theorem 4.4.2]{HormanderBook} will be
used.
\begin{theorem*}[H\"{o}rmander]  
Let $\D$ be a pseudoconvex domain in $\C^{n}$ and $\psi$ be a continuous
plurisubharmonic function on $\D$. Assume that 
$u =\sum_{j=1}^{n}u_{j}d\bar z_{j} \in L^{2}_{(0,1)}(\D,e^{-\psi})$
such that $\dbar u=0$. Then there exists  $f \in L^{2}(\D,e^{-\psi})$
such that $\dbar f=u$ and 
\[\int_{\D}\frac{|f(z)|^{2}}{(1+\sum_{j=1}^n|z_j|^{2})^{2}}e^{-\psi(z)}d\lambda(z)
\leq \int_{\D}\sum_{j=1}^{n}|u_{j}(z)|^{2}e^{-\psi(z)}d\lambda(z)\]
where $z=(z_1,\dots,z_n)\in \C^n.$
\end{theorem*}

We  include the following  standard Lemma and its proof for convenience of the
reader. 
\begin{lemma} \label{Lem2} 
Let $\D$ be a bounded pseudoconvex domain in $\C^n, B(p,r)$ be the ball
centered at $p\in b\D$ with radius $r,$ and $\D(p,r)=B(p,r)\cap \D.$ 
For $\varepsilon>0$ and $0<\delta<r$ there exists a bounded operator
$E_{\varepsilon,\delta}:A^2(\D(p,r))\to A^{2}(\D)$ such that  
\begin{equation*}
\|f-E_{\varepsilon,\delta}(f)\|_{L^{2}(\D(p,r-\delta))} \leq \varepsilon
\|f\|_{L^{2}(\D(p,r-\delta))} \text{ for } f\in A^{2}(\D(p,r)). 
\end{equation*}

\end{lemma}
The following proof will use H\"{o}rmander's Theorem in a similar fashion as in the 
proof of \cite[TheoremVI.3]{JupiterThesis} where Jupiter shows that a pseudoconvex
domain in $\C^n$ is a Runge domain  if and only if it is polynomially convex.

\begin{proof}[Proof of Lemma \ref{Lem2}]
The crucial step in the proof is constructing a sequence of weight functions that will
allow us to get the desired norm estimates. To that end, let us choose positive numbers
$\delta,r_1,$ and $r_2$ so that $0<r-\delta=r_{1}<r_{2}<r$ and define a function $\psi$ 
as 
\[\psi(z)=-r_{2}^{2}+\sum_{j=1}^n|z_j-p_j|^{2}\] 
 where $z=(z_1,\ldots,z_n)\in \C^n.$ 
Furthermore, we choose a smooth cut-off function $\chi\in C^{\infty}_{0}(B(p,r))$   such
that $\chi \equiv 1 $ in a neighborhood of  $\overline{B(p,r_{2})}.$ We note
that $\psi$ is a continuous plurisubharmonic function  on $\C^{n}$ that satisfies 
the following crucial property: $\psi(z)<0$ for  $z\in B(p,r_{2})$ and $\psi(z)>0$ for
$z\in \C^{n}\setminus \overline{B(p,r_{2})}.$ Since $\psi$ is bounded on $\D,$ the Hilbert
spaces  $L^{2}(\D)$ and $L^{2}(\D, e^{-k\psi})$ are equal for all $k$ as sets. Then 
H\"{o}rmander's Theorem  implies that for every $k$ there exists
$u_{k}\in L^{2}(\D)$ such that $\dbar u_{k}=f\dbar \chi$ with 
\begin{align}\label{Eqn1}
\int_{\D}|u_{k}(z)|^{2}e^{-k\psi(z)}d\lambda(z)\leq 
C\int_{\D}|f(z)|^{2}\sum_{j=1}^{n}\left|\frac{\partial
\chi(z)}{\partial \bar z_{j}}\right|^{2} e^{-k\psi(z)} d\lambda(z)
\end{align} 
where $C$ is a positive real number that depends only on $\D.$  
We note that  $\psi <- r_2^2+r_1^2<0$ on $B(p,r_{1})$ and $\psi$ is strictly
positive on a neighborhood of the support of the $\dbar \chi$. Hence the right hand side
of \eqref{Eqn1} goes to zero as $k$ goes to infinity and we have 
\begin{align}\nonumber 
\int_{\D\cap B(p,r_{1})}|u_{k}(z)|^{2} d\lambda(z) &\leq
\int_{\D}|u_{k}(z)|^{2}e^{-k\psi(z)}d\lambda(z) \\
\label{Eqn2}&\leq C \int_{\D}|f(z)|^{2}\sum_{j=1}^{n}\left|
\frac{\partial \chi(z)}{\partial \bar z_{j}}\right|^{2}
e^{-k\psi(z)}d\lambda(z). 
\end{align}
Then  depending on $\varepsilon$ and $\delta$ (and using \eqref{Eqn2}) we can choose
$C_{\varepsilon,\delta}>0$ and $k$  so that 
$\|u_{k}\|_{L^{2}(\D(p,r_1))} \leq \varepsilon \|f\|_{L^{2}(\D(p,r_1))}$
and $\|u_{k}\|_{L^{2}(\D)} \leq C_{\varepsilon,\delta}\|f\|_{L^{2}(\D(p,r))}.$
Therefore, we can define  $E_{\varepsilon,\delta}$  as 
$E_{\varepsilon,\delta}(f)=\chi f-u_{k}.$ 
\end{proof}

Now we are ready to prove Theorem \ref{ThmMain}.

\begin{proof}[Proof of Theorem \ref{ThmMain}] 
To simplify the notation in this proof we will denote the norm $\|.\|_{L^2(U)}$ 
by $\|.\|$ and  the operator $H^U_{R_U(\phi)}$ by $H^U_{\phi}.$  We note that
$\langle .,.\rangle $ denotes the inner product on $U$ and $A\lesssim B$ means
that $A\leq cB$ for some constant $c$ that is independent of the parameters of interest
and its value can change at every appearance. For $f\in A^{2}(U)$ we have
 \begin{align*}
 \| H^{U}_{\phi} (f)\|^{2}=&\langle \dbar^{*}N^U(f\dbar\phi),\dbar^{*}N^U(f\dbar
\phi) \rangle \\
 =&\langle f\dbar\phi  ,N^U\dbar\dbar^{*}N^U(f\dbar\phi)  \rangle \\
=&\langle f\dbar\phi, N^U(f\dbar\phi)\rangle. 
\end{align*}
In the last equality above we used the facts that
$N^U(\dbar\dbar^{*}+\dbar^{*}\dbar)=I$ and $\dbar N^U \dbar=0.$ Now we will construct a
smooth bounded function $\lambda$ that has a large Hessian on the boundary of the ball
$B(p,r).$  Let $\gamma:\mathbb{R}\to\mathbb{R}$ be a smooth, non-decreasing, convex
function  such that $-1\leq \gamma(t)\leq 0$ for $t\leq 0,\gamma(0)=0,$ and
$\gamma'(0)\geq 2.$ Furthermore, let us define 
\[\rho_{\varepsilon}(z)=\frac{1}{\varepsilon}\left(-r^2+\sum_{j=1}^n|z_j-p_j|^2\right)\]
for  $r,\varepsilon>0$  and $\psi_{\varepsilon} (z)= \gamma(\rho_{\varepsilon}(z)).$ 
Then one can check that $\psi_{\varepsilon}$ is a smooth plurisubharmonic function on
$\C^n,$  such that $-1\leq\psi_{\varepsilon}(z) \leq  0$ for $z\in B(p,r).$ 
Also, by continuity, there exists $\delta>0$ such that 
\[\sum_{j,k=1}^n \frac{\partial^2 \psi_{\varepsilon} (z)}{\partial z_j\partial \overline
z_k} w_j\overline w_k \geq \frac{1}{\varepsilon}\sum_{j=1}^n|w_j|^{2}\] 
for $z\in K=\overline{B(p,r)\setminus B(p,r-\delta)}$ and 
$(w_1,\ldots w_n)\in \C^{n}.$ 
Then (ii) in \cite[Corollary 2.13]{StraubeBook} implies that 
\begin{align} \nonumber 
\frac{1}{e \varepsilon}\int_{K\cap U}|h(z)|^{2}d\lambda(z) &\leq \sum_{j,k=1}^n
\int_{U}e^{\psi_{\varepsilon}(z)} \frac{\partial^2 \lambda (z)}{\partial z_j\partial
\overline z_k} h_j(z)\overline{h_k(z)}d\lambda(z)\\
\label{EqnMain}& \leq \|\dbar h\|^{2}+\|\dbar^{*}h\|^{2} 
\end{align}
for $ h=\sum_{j=1}^nh_jd\overline z_j \in Dom(\dbar)\cap Dom(\dbar^{*}) \subset
L^{2}_{(0,1)}(U).$ Let $\chi\in C^{\infty}(\overline{B(p,r)})$ such that
$\chi\equiv 1$ on a neighborhood of  $bB(p,r),$ and $\chi\equiv 0$ on
$B(p,r-\delta).$ Then 
\begin{align*}
\|H^{U}_{\phi} (f)\|^{2}\leq & |\langle f\dbar\phi, \chi N^U(f\dbar\phi)\rangle|
+|\langle f\dbar\phi, (1-\chi)N^U(f\dbar\phi)\rangle| \\
\leq& \|f\dbar\phi\|  \|\chi N^U(f\dbar\phi)\|+|\langle (1-\chi)f\dbar\phi,
N^U(f\dbar\phi)\rangle|.
\end{align*}
Then \eqref{EqnMain} implies that
\begin{align*}
\|\chi N^U(f\dbar \phi)\|^2 &\lesssim \varepsilon \left(\|\dbar N^U(f\dbar
\phi)\|^2+\|\dbar^{*} N^U(f\dbar \phi)\|^2\right)\\
& \lesssim  \varepsilon \|f\|^2 
\end{align*}
for $ f\in A^2(U).$ Let us denote $\chi_{1}=1-\chi$ and choose  
$\widetilde{\chi} \in C^{\infty}_{0}(B(p,r))$ such that  $0 \leq \widetilde{\chi} \leq 1 $
and $\widetilde{\chi}\equiv 1$ on the support of $\chi_1.$ Then Lemma \ref{Lem2}
implies that there exists a bounded operator  
$E_{\varepsilon,\delta}:  A^{2}(U)\to A^{2}(\D)$ such that 
$\|\widetilde{\chi}(R_UE_{\varepsilon,\delta}(f)-f)\|\leq \varepsilon \|f\|.$
Since $\delta$ depends on $\varepsilon$ in the following calculation we
will use the following notation: 
$E_{\varepsilon}=E_{\varepsilon,\delta},F_{\varepsilon}=E_{\varepsilon}(f).$ 
Let $M_{\varepsilon}$ denote the norm of the operator $E_{\varepsilon}.$

We note that in the following inequalities $\dbar^*_{\D}$ and $\dbar^*$ denote
the Hilbert space adjoints of $\dbar$ on $\D$ and on $U$, respectively. A
$(0,1)$-form $f$ is in the domain of $\dbar^*$ if there exists a square integrable
function $g$ such that $\langle f,\dbar h\rangle=\langle g,h\rangle $ for all $h$ in the
domain of $\dbar.$ Furthermore, if a $(0,1)$-form $f=\sum_{j=1}^nf_jd\overline{z}_j$ is in
the domain of $\dbar^*$ then $\dbar^*f=-\sum_{j=1}^n \frac{\partial f_j}{\partial z_j}$
in the  sense of distributions (see Chapter 4.2 in \cite{ChenShawBook} for more
information). The fact that $\dbar^*N$ is a solution operator
for $\dbar$ (that is, $\dbar\dbar^*Nf=f$ if $f$ is a $\dbar$-closed form) implies
that $ F_{\varepsilon} \dbar \phi=\dbar(F_{\varepsilon}\phi)=\dbar
\dbar^*_{\D}N^{\D}F_{\varepsilon} \dbar\phi.$ We will use this equality  as
well as the Cauchy-Schwarz inequality to pass from the first line to the second
line below.
\begin{align*}
 |\langle \chi_{1}(f\dbar\phi),  N^U(f\dbar\phi) \rangle|
&\leq |\langle \chi_{1}  (f-F_{\varepsilon})\dbar\phi, N^U(f\dbar\phi)\rangle|
+|\langle \chi_{1}   F_{\varepsilon}\dbar \phi, N^U(f\dbar\phi)\rangle|   \\
&\lesssim \|\chi_{1} (f-F_{\varepsilon})\|\|f\| +|\langle \chi_{1}
\dbar\dbar^{*}_{\D}N^{\D} (F_{\varepsilon}\dbar\phi), N^U(f\dbar\phi)\rangle| \\
&\lesssim  \|\widetilde{\chi} (f-F_{\varepsilon})\|\|f\|+ |\langle \dbar^{*}_{\D}
N^{\D} (F_{\varepsilon}\dbar\phi),\dbar^{*} \chi_{1} N^U(f\dbar\phi)\rangle|  \\
&\lesssim \varepsilon\|f\|^{2} + \widetilde{C}_{\varepsilon}
\|\dbar^{*}_{\D}N^{\D}(F_{\varepsilon}\dbar\phi) \|_{L^{2}(\D)}\|f\|,
\end{align*}
where $\widetilde{C}_{\varepsilon}$ is a constant that is independent of $f.$
Now we will use the fact that $H_{\phi}^{\D}$ is compact on $A^2(\D)$ and
$\|F_{\varepsilon}\|_{_{L^{2}(\D)}}\leq M_{\varepsilon}\|f\|_{L^{2}(U)} .$ Lemma
\ref{LemCompactH} implies that for any $\varepsilon'>0 $ there exists a compact
operator $K_{\varepsilon'}$ on $X_{\phi}(\D)$ such that
\[\|\dbar^{*}_{\D}N^{\D}(F_{\varepsilon}\dbar\phi) \|_{L^{2}(\D)} \lesssim
\varepsilon' \|F_{\varepsilon}\|_{L^{2}(\D)}+\|K_{\varepsilon'} \Pi_{\dbar\phi}
(F_{\varepsilon})\|_{L^{2}(\D)}.\]   
where $\Pi_{\dbar\phi}:A^2(\D)\to X_{\phi}(\D)$ denotes the (bounded)
multiplication operator by $\dbar\phi.$ That is, $\Pi_{\dbar\phi} h =
h\dbar\phi$ for $h\in
A^2(\D)$. Therefore, for $f\in A^2(U)$ we have the following inequality
\begin{align*}
 \|H^{U}_{\phi} (f)\|^{2} \lesssim& \left(
\varepsilon+\sqrt{\varepsilon} +
\varepsilon'M_{\varepsilon}\widetilde{C}_{\varepsilon} \right)\|f\|^2+
\widetilde{C}_{\varepsilon} \| f\|
\|K_{\varepsilon'}\Pi_{\dbar\phi}E_{\varepsilon}(f) \|_{L^2(\D)} \\
\leq &\left(
\varepsilon+\sqrt{\varepsilon} +
\varepsilon'M_{\varepsilon}\widetilde{C}_{\varepsilon}+
\varepsilon'\widetilde{C}_{\varepsilon}\right)\|f\|^2\\
&+ \left(\widetilde{C}_{\varepsilon} 
+\frac{\widetilde{C}_{\varepsilon}}{\varepsilon'}\right)
\|K_{\varepsilon'}\Pi_{\dbar\phi}E_{\varepsilon}(f)\|^2_{L^2(\D)}.
\end{align*}
 For any $0<\varepsilon<1$ there exists $\varepsilon'>0$ so that 
$\varepsilon+\sqrt{\varepsilon} +
\varepsilon'M_{\varepsilon}\widetilde{C}_{\varepsilon} \leq
2\sqrt{\varepsilon}.$ Then  the above inequality combined with fact that 
$x^2+y^2\leq (x+y)^2$ for $x,y\geq 0$  imply the following:  for any
$0<\varepsilon<1$ there exists a compact operator
$K_{\varepsilon}=(\widetilde{C}_{\varepsilon}+\widetilde{C}_{\varepsilon}/
\varepsilon')^{1/2}K_{\varepsilon'}\Pi_{\dbar\phi}E_{\varepsilon}$ such that 
\[\|H^U_{\phi}(f)\| \lesssim \varepsilon^{1/4} \|f\| + 
\|K_{\varepsilon}(f)\| \text{ for } f\in A^2(U).\] 
Now Lemma \ref{LemCompEstimate} implies that $H^{U}_{\phi}$ is compact on
$A^{2}(U).$ 
\end{proof}

\begin{proof}[Proof of Proposition \ref{PropInC}]
Since functions that are smooth up to the boundary of $\D$ are dense in $C(\Dc)$ and the
sequence  $\{H^{\D}_{\psi_n}\}$ converges to $H^{\D}_{\psi}$ in the operator norm
whenever $\{\psi_n\}$ converges to $\psi$ uniformly on $\Dc$  it suffices to prove that
$H^{\D}_{\psi}$ is compact whenever $\psi\in C^{\infty}(\Dc).$ Let us define 
\[S_{\psi} (f)(z)= -\frac{1}{\pi }\int_{\D}
\frac{\frac{\partial\psi}{\partial\overline\xi}(\xi) f(\xi)}{\xi-z}d\lambda(\xi)\] 
for $f\in A^2(\D)$ and $z\in \D.$ We will show that $H^{\D}_{\psi}$ is compact
on $ A^2(\D)$ by  showing that $S_{\psi}$ is a limit of compact operators (in
the operator norm) and  $ S_{\psi}(f)$ solves $\dbar u =f\dbar\psi$ 
(because $H^{\D}_{\psi} =S_{\psi} - P^{\D}S_{\psi}$). To that end, for $\varepsilon>0$ 
let $\chi_{\varepsilon}$ be a smooth cut-off function on $\mathbb{R}$ such that
$\chi_{\varepsilon}\equiv 1$ on a neighborhood of the origin and
$\chi_{\varepsilon}(t)=0$ for $|t|\geq \epsilon.$ Then
$S_{\psi}=A^{\varepsilon}_{\psi}+B^{\varepsilon}_{\psi}$ where 
\begin{align*}
 A^{\varepsilon}_{\psi}(f)(z)&=-\frac{1}{\pi
}\int_{\D}\frac{\chi_{\varepsilon}(|\xi-z|)\frac{\partial\psi}{\partial \overline\xi}(\xi)
f(\xi)}{\xi-z}d\lambda(\xi)\\
 B^{\varepsilon}_{\psi}(f)(z)&=-\frac{1}{\pi}\int_{\D}
\frac{(1-\chi_{\varepsilon} (|\xi-z|))
\frac{\partial\psi}{\partial \overline\xi}(\xi) f(\xi)}{\xi-z}d\lambda(\xi).
\end{align*}
Then  the operator $B^{\varepsilon}_{\psi}$ is Hilbert-Schmidt and, in particular, 
compact because the kernel
\[-\frac{(1-\chi_{\varepsilon}(|\xi-z|))\frac{\partial\psi}{\partial
\overline\xi}(\xi)}{\pi(\xi-z)}\]
 is square integrable on $\D\times \D.$ 

Next we will show that $A^{\varepsilon}_{\psi}$ has a small norm. Let $\widehat{f}$ denote
the trivial extension of $f$. That is, $\widehat{f}=f$ on $\D$ but  $\widehat{f}=0$
otherwise. Since $\frac{\partial\psi}{\partial \overline\xi}$ is continuous on $\Dc$ and
$\D$ is bounded, using polar coordinates, we get 
\[ |A^{\varepsilon}_{\psi}(f)(z)|\lesssim  \int_{\C}
\frac{|\chi_{\varepsilon}(|\xi|)\widehat{f}(z+\xi)|}{|\xi|} d\lambda(\xi) 
\lesssim \int_0^{2\pi}\int_0^{\varepsilon}|\widehat{f}(z+re^{i\theta})| dr d\theta.\]
Then the Cauchy-Schwarz inequality together with Fubini's theorem yield that
 \begin{align*}
 \|A^{\varepsilon}_{\psi}(f)\|^2\lesssim 2\pi \varepsilon
\int_0^{2\pi}\int_0^{\varepsilon}\int_{\D}|\widehat{f}(z+re^{i\theta})|^2d\lambda(z) dr
d\theta \leq 4\pi^2\varepsilon^2\|f\|^2.
\end{align*}
Hence,  $\|A^{\varepsilon}_{\psi}\|\lesssim \varepsilon$ and $S_{\psi}$ is a limit (in
the operator norm) of a sequence $\{B^{1/k}_{\psi}\}$ of compact operators.

Next we want to show that $\dbar S_{\psi}(f) =f\dbar\psi.$ Let $\{f_n\}$ be a
sequence of functions that are smooth on $\Dc$ and converging to $f$ in
$L^2(\D).$ Then the Cauchy integral with remainder formula 
(see \cite[Theorem 2.1.2]{ChenShawBook}) shows that 
$\dbar S_{\psi}(f_n)=f_n\dbar\psi.$ On the other hand, $\{\dbar S_{\psi}(f_n)\}$
converges weakly to $\dbar S_{\psi}(f)$ and $\{f_n\dbar\psi\}$ converges to
$f\dbar\psi$ in $L^2(\D).$ Therefore, $\dbar S_{\psi}(f) =f\dbar\psi $ for $f\in A^2(\D).$
\end{proof}

\section*{Acknowledgement}
I am in debt to  Daniel Jupiter for bringing the proof of 
\cite[Theorem VI.3]{JupiterThesis} to my attention, to Mehmet \c{C}elik and Trieu Le
for fruitful conversations, to my advisor Emil Straube for the proof of Proposition
\ref{PropInC}, and to the referee for valuable comments.



\begin{thebibliography}{Koh63}
\bibitem[CD97]{CatlinD'Angelo97}
David~W. Catlin and John~P. D'Angelo, \emph{Positivity conditions for
  bihomogeneous polynomials}, Math. Res. Lett. \textbf{4} (1997), no.~4,
  555--567.

\bibitem[{\c{C}}{\c{S}}]{CelikSahutoglu}
Mehmet {\c{C}}elik and S\"{o}nmez {\c{S}}ahuto\u{g}lu, \emph{On compactness of
  the $\overline{\partial}$-{N}eumann problem and {H}ankel operators}, to
  appear in Proc. Amer. Math. Soc.

\bibitem[CS01]{ChenShawBook}
So-Chin Chen and Mei-Chi Shaw, \emph{Partial differential equations in several
  complex variables}, AMS/IP Studies in Advanced Mathematics, vol.~19, American
  Mathematical Society, Providence, RI, 2001.

\bibitem[{\v{C}}{\c{S}}09]{CuckovicSahutoglu09}
{\v{Z}}eljko {\v{C}}u{\v{c}}kovi{\'c} and S{\"o}nmez {\c{S}}ahuto{\u{g}}lu,
  \emph{Compactness of {H}ankel operators and analytic discs in the boundary of
  pseudoconvex domains}, J. Funct. Anal. \textbf{256} (2009), no.~11,
  3730--3742.

\bibitem[D'A02]{D`AngeloIneqBook}
John~P. D'Angelo, \emph{Inequalities from complex analysis}, Carus Mathematical
  Monographs, vol.~28, Mathematical Association of America, Washington, DC,
  2002.

\bibitem[FS98]{FuStraube98}
Siqi Fu and Emil~J. Straube, \emph{Compactness of the
  {$\overline\partial$}-{N}eumann problem on convex domains}, J. Funct. Anal.
  \textbf{159} (1998), no.~2, 629--641.

\bibitem[H{\"o}r65]{Hormander65}
Lars H{\"o}rmander, \emph{{$L\sp{2}$} estimates and existence theorems for the
  {$\bar \partial $}\ operator}, Acta Math. \textbf{113} (1965), 89--152.

\bibitem[H{\"o}r90]{HormanderBook}
\bysame, \emph{An introduction to complex analysis in several variables}, third
  ed., North-Holland Mathematical Library, vol.~7, North-Holland Publishing
  Co., Amsterdam, 1990.

\bibitem[Jup03]{JupiterThesis}
Daniel Jupiter, \emph{Envelopes of holomorphy and approximation theorems},
  Ph.D. thesis, University of Michigan, Ann Arbor, MI, 2003.

\bibitem[Koh63]{Kohn63}
J.~J. Kohn, \emph{Harmonic integrals on strongly pseudo-convex manifolds. {I}},
  Ann. of Math. (2) \textbf{78} (1963), 112--148.

\bibitem[Kra01]{KrantzBook}
Steven~G. Krantz, \emph{Function theory of several complex variables}, AMS
  Chelsea Publishing, Providence, RI, 2001, Reprint of the 1992 edition.

\bibitem[Ran86]{RangeBook}
R.~Michael Range, \emph{Holomorphic functions and integral representations in
  several complex variables}, Graduate Texts in Mathematics, vol. 108,
  Springer-Verlag, New York, 1986.

\bibitem[Str10]{StraubeBook}
Emil~J. Straube, \emph{Lectures on the $\mathcal{L}^{2}$-{S}obolev theory of
  the $\overline\partial$-{N}eumann problem}, ESI Lectures in Mathematics and
  Physics, vol.~7, European Mathematical Society (EMS), Z\"urich, 2010.

\bibitem[Str11]{StraubePrivate}
\bysame, Private communication, 2011.

\bibitem[Zhu07]{ZhuBook}
Kehe Zhu, \emph{Operator theory in function spaces}, second ed., Mathematical
  Surveys and Monographs, vol. 138, American Mathematical Society, Providence,
  RI, 2007.
\end{thebibliography}
\end{document}